\documentclass[a4paper, 12pt, dvipdfmx]{amsart}

\usepackage{amssymb, amsmath, amsthm}

\theoremstyle{plain}
\newtheorem{theorem}{Theorem}[section]

\newtheorem{lemma}[theorem]{Lemma}
\newtheorem{conjecture}[theorem]{Conjecture}
\newtheorem{corollary}[theorem]{Corollary}
\newtheorem{proposition}[theorem]{Proposition}
\newtheorem{definition}[theorem]{Definition}
\newtheorem{remark}[theorem]{Remark}

\makeatletter

\@addtoreset{equation}{section}
\makeatother

\title[Minimax values of $p$-energy and packing radii]{
Certain min-max values related to the $p$-energy and packing radii of Riemannian manifolds and metric measure spaces}
\author{Ayato Mitsuishi}
\email[A.~Mitsuishi]{mitsuishi@fukuoka-u.ac.jp}
\date{\today}

\newcommand{\Lip}{\mathrm{Lip}}
\newcommand{\lip}{\mathrm{lip}}
\newcommand{\pack}{\mathrm{pack}}

\newcommand{\inrad}{\mathrm{inrad}}
\newcommand{\inpack}{\mathrm{inpack}}

\def\Xint#1{\mathchoice
{\XXint\displaystyle\textstyle{#1}}%
{\XXint\textstyle\scriptstyle{#1}}%
{\XXint\scriptstyle\scriptscriptstyle{#1}}%
{\XXint\scriptscriptstyle\scriptscriptstyle{#1}}%
\!\int}
\def\XXint#1#2#3{{\setbox0=\hbox{$#1{#2#3}{\int}$}
\vcenter{\hbox{$#2#3$}}\kern-.5\wd0}}

\def\dashint{\Xint-}

\begin{document}
\maketitle

\begin{abstract}
Grosjean proved that the $(1/p)$-th power of the first eigenvalue of the $p$-Laplacian on a closed Riemannian manifold converges to the twice of the inverse of the diameter of the space, as $p \to \infty$. 
Before this, a corresponding result for the Dirichlet first eigenvalues was also obtained by Juutinen, Lindqvist and Manfredi. 
We extend those results for certain $k$-th min-max value related to the $p$-energy, 
where the corresponding limits are packing radii introduced by Grove-Markvorsen or its variant. 
Furthermore, we remark that our result holds for more singular setting. 
\end{abstract}

\section{Introduction and a main result} \label{sec:intro}

\subsection{A main result}
Let $p > 1$. 
The $p$-Laplacian on a closed Riemannian manifold $M$ is defined by 
\[
\triangle_p u = - \mathrm{div} (|\nabla u|^{p-2} \nabla u)
\]
for smooth functions $u : M \to \mathbb R$. 
The operator $\triangle_p$ is well-defined on the set $W^{1,p}(M)$ of all $(1,p)$-Sobolev functions, and is non-linear if $p \ne 2$. 
We say that $\lambda \ge 0$ is an {\it eigenvalue} of $\triangle_p$ if 
\[
\triangle_p u = \lambda |u|^{p-2} u
\]
holds for some non-trivial function $u \in W^{1,p}(M)$ in the weak sense.
Such a $u$ is called an eigenfunction of $\triangle_p$ for $\lambda$. 
The first non-zero eigenvalue of $\triangle_p$ is denoted by $\lambda_{1,p}(M)$. 
About this value, Grosjean proved
\begin{theorem}[\cite{Gros}] \label{thm:Gros}
If $M$ is a closed Riemannian manifold, then 
\[
\lim_{p \to \infty} \lambda_{1,p}(M)^{1/p} = \frac{2}{\mathrm{diam}(M)}.
\]
Here, $\mathrm{diam}(M)$ stands for the diameter of $M$, that is, $\max_{x,y \in M} |x,y|$.
\end{theorem}
This is generalized to more singular metric measure spaces (\cite{Ho}, \cite{AH}). 
Before this result, the Dirichlet first eigenvalue case was proved by Juutinen, Lindqvist and Manfredi (\cite{JLM}, see also Lemma \ref{lem:JLM}).

In this paper, we consider variants of the diameter and the first eigenvalue $\lambda_{1,p}$ of the $p$-Laplacian as follows. 
Replacements of the diameter are the {\it packing radii} introduced by Grove-Markvorsen (\cite{GM1}, \cite{GM2}): 
\[
\mathrm{pack}_{k+1} (M) := \frac{1}{2} \max_{x_0, x_1, \dots, x_k \in M} \min_{i \ne j} |x_i, x_j|
\]
for $k \ge 1$. 
It is the largest $r > 0$ such that $M$ can contain disjoint $k+1$ open balls of radius $r$. 
Note that $\mathrm{pack}_2 = \frac{1}{2} \mathrm{diam}$. 
The sequence $\{\mathrm{pack}_{k+1}\}_k$ is non-increasing in $k$ and goes to zero as $k \to \infty$. 
As a replacement of $\lambda_{1,p}$, 
we introduce 
a kind of min-max value defined as 
\begin{equation} \label{eq:over,k,p}
\overline \lambda_{k,p}(M) :=
\inf_{A_0, A_1, \dots, A_k : \text{ disjoint}} \max_{i} \lambda_{1,p}^D(A_i). 
\end{equation}
Here, $A_i$ denote mutually disjoint non-empty open subsets in $M$, and 
\begin{equation} \label{eq:1st-D}
\lambda_{1,p}^D(A) := \inf \left\{ \frac{\|\nabla f\|_p^p}{\|f\|_p^p} \,\middle|\, f \in W^{1,p}_0(A) \setminus \{0\}
\right\}
\end{equation}
is {\it the first Dirichlet eigenvalue of the $p$-Laplacian} on 
a non-empty proper open subset 
$A \subset M$. 
Values similar to $\overline \lambda_{k,p}$ using separation of the space are studied in \cite{Mic}, \cite{CL}. 

A main result in the paper is: 
\begin{theorem} \label{thm:1}
If $M$ is a closed Riemannian manifold, then we have 
\[
\lim_{p \to \infty} \overline \lambda_{k,p}(M)^{1/p} = \mathrm{pack}_{k+1}(M)^{-1}.
\]
\end{theorem}
It is known that $\overline \lambda_{1,p}$ coincides with the first eigenvalue $\lambda_{1,p}$ of the $p$-Laplacian (\cite{Veron}, \cite[Lemmas 3.1 and 3.2]{Matei}, \cite[Lemma 1.9.5]{AH}).
So, our Theorem \ref{thm:1} is regarded as a ``$k$-th version'' of Theorem \ref{thm:Gros}. 
A result similar to Theorem \ref{thm:1} was obtained for the Dirichlet eigenvalues on bounded domains of Euclidean spaces (\cite[Theorem 4.1]{JL}) for $k=2$. 
The author do not know whether $\overline \lambda_{k,p}$ is an eigenvalue of $\triangle_p$ or not, even if $p=2$. 
So, we often call the sequence $\{\overline \lambda_{k,p}\}_k$ a {\it fake spectrum}.
For another value similar to $\overline \lambda_{k,p}$, we also obtain a result similar to Theorem \ref{thm:1} (Corollary \ref{cor:1}). 
After proving Theorem \ref{thm:1} and Corollary \ref{cor:1}, we give a remark that the statements of them hold for more general metric measure spaces (Theorem \ref{thm:2}). 

\vspace{1em}
\noindent{\bf Organization}.
In \S \ref{sec:proof}, we prove Theorem \ref{thm:1} and some fundamental properties of our fake spectrum.
Furthermore, we consider a variant of $\overline \lambda_{k,p}$ and a Dirichlet boundary problem version of $\overline \lambda_{k,p}$. 
We compare fake and real spectra of the $p$-Laplacian. 
In \S \ref{sec:gen}, we remark that our results are generalized to more singular metric measure spaces (Theorem \ref{thm:2}). 
Furthermore, in there, we give an example satisfying the assumption of Theorem \ref{thm:2}, but which does not satisfies any curvature-dimension condition. 
In \S \ref{sec:Weyl}, we state a conjecture about an asymptotic law of packing radii related with a recent Mazurowski's asymptotic law of a spectrum of the $p$-Laplacian.
In Appendix \ref{sec:Morrey}, we give a proof of Theorem \ref{thm:HK} to complete the proof of Theorem \ref{thm:2}. 

\vspace{1em}
\noindent{\bf Acknowledgements}. 
The author would like to express my appreciation to Professors Kei Funano, Shouhei Honda and Yu Kitabbepu for valuable comments and discussions. 
In particular, K.~Funano told me several literatures related to our work, and S.~Honda taught me a discussion and literatures on analysis on Riemannian manifolds/metric measure spaces. 
This work was supported by JSPS KAKENHI 17H01091.

\section{Proof of Theorem \ref{thm:1} and several properties} \label{sec:proof}
Let $M = (M,g)$ denote a closed Riemannian manifold with a Riemannian metric $g$. 
We denote by $m$ the normalized volume measure $m := \mathrm{vol}_g / \mathrm{vol}_g(M)$, where $\mathrm{vol}_g$ is the standard volume measure of $(M,g)$. 
Since $\overline \lambda_{k,p}$ is invariant under multiplication of the measure with positive constant, 
the normalization of $m$ is not important. 

\subsection{Notation} \label{sec:notation}
We fix the notation.
Let $k$ be a positive integer and $p > 1$ a real number. 
For $x,y \in M$, $|x,y|$ stands for the distance between $x$ and $y$. 
For $A \subset M$ and $x \in M$, $|A,x| = |x,A| := \inf_{a \in A} |a,x|$. 
For $x \in M$ and $r > 0$, $U_r(x) := \{y \in M \mid |x,y| < r\}$ denotes the open $r$-ball around $x$. 
For a measurable function $f : M \to \mathbb R$, $\|f\|_p = (\int_M |f|^p\,dm)^{1/p}$ denotes the $p$-norm of $f$ with respect to $m$.
If $f$ has the weak derivative $\nabla f$, then the $(1,p)$-norm of $f$ is defined as 
\[
\|f\|_{1,p} := (\|f\|_p^p + \|\nabla f\|_p^p)^{1/p} \in [0, \infty].
\]
Let $W^{1,p}(M)$ denote the subspace of $L^p(M)$ with $\|f\|_{1,p} < \infty$.
For an open subset $\Omega$ in a complete Riemannian manifold, $W_0^{1,p}(\Omega)$ denotes the $W^{1,p}$-closure of the space of all Lipschitz functions which have the compact support in $\Omega$. 

Let $\Lip(M)$ be the set of all Lipschitz functions on $M$ and $\Lip(f)$ the Lipschitz constant of $f$.  
Note that for a Lipschitz function $f : M \to \mathbb R$, we have 
\begin{equation} \label{eq:Lip}
\Lip(f) = \| \nabla f\|_\infty. 
\end{equation}

\subsection{Lemmas and a proof of Theorem \ref{thm:1}}
For a non-empty bounded open subset $\Omega$ of a complete Riemannian manifold with $\partial \Omega \ne \emptyset$, we define the {\it inradius} of $\Omega$ as
\[
\mathrm{inrad}(\Omega) := \max_{x \in \Omega} |x, \partial \Omega|. 
\]
Here, $\partial \Omega$ stands for the topological boundary of $\Omega$.
This is the maximal radius of which $\Omega$ can contain a metric ball. 

\begin{lemma} \label{lem:open-Lip}
Let $\Omega$ be a bounded open subset in a complete Riemannian manifold with $\partial \Omega \ne \emptyset$.
Let $f$ be a Lipschitz function on the space which is zero outside $\Omega$ with $\|f\|_\infty = 1$. 
Then, we have 
\[
\inrad(\Omega)^{-1} \le \Lip(f).
\]
\end{lemma}
\begin{proof}
Since $\|f\|_\infty = 1$, for any $\delta > 0$, there exists $x \in \Omega$ such that $|f(x)| > 1 - \delta$. 
Let us take $y \in \partial \Omega$ with $|x,y| = |x, \partial \Omega|$. 
From the assumption, we have $f(y)=0$. 
Hence, we obtain 
\[
1-\delta < |f(x)-f(y)| \le \Lip(f) |x,y| \le \Lip(f) \inrad(\Omega). 
\]
Letting $\delta \to 0$, we obtain the conclusion. 
\end{proof}

\begin{lemma}[{\cite[Lemma 1.5]{JLM}}] \label{lem:JLM}
For a bounded open subset $A$ of a complete Riemannian manifold with $\partial A \ne \emptyset$, we have 
\[
\lim_{p \to \infty} \lambda_{1,p}^D(A)^{1/p} = \mathrm{inrad}(A)^{-1}. 
\]
\end{lemma}
To use an argument of the proof of Lemma \ref{lem:JLM} later, we only give a proof of the lim-inf inequality. 
Suppose that the lim-sup inequality holds. 
Let $\epsilon > 0$. 
Let us take $u_p \in W^{1,p}_0(A)$ satisfying $\|u_p\|_p=1$ and 
\begin{equation} \label{eq:01}
\|\nabla u_p\|_p \le \lambda_{1,p}(A)^{1/p} + \epsilon. 
\end{equation}
By the lim-sup inequality, $\sup_{p > p_0} \|\nabla u_p\|_p < \infty$ for some $p_0 >1$. 
Due to Morrey's inequality (for instance, see \cite[Theorem 3 in p. 143]{EG}, \cite[Theorem 9.12]{Br}), 
we know that 
\begin{equation} \label{eq:Morrey}
|u_p(x)-u_p(y)| \le C_{p,n,d,\kappa} |x,y|^{1-n/p} \|\nabla u_p\|_p
\end{equation}
where $C_{p,n,d, \kappa}$ is a constant depend only on $p, n, d$ and $\kappa$, $n = \dim M$, $d = \mathrm{diam}(M)$ and $\kappa$ denotes the lower Ricci curvature bound of $M$. 
Note that 
\begin{equation} \label{eq:uniform Holder}
\sup_{p >p_0} C_{p,n,d, \kappa} < \infty
\end{equation}
holds.  
See for instance \cite[p.283 (28)]{Br} for the Euclidean case.
For a general case, we will verify in Appendix \ref{sec:Morrey}. 
Therefore, $u_p$ has a uniformly H\"older continuous representative.
Further, by \eqref{eq:uniform Holder}, we have 
\[
\sup_{x \in A}|u_p(x)| \le C_{p,n,d,\kappa} \inrad(A)^{1-n/p} \|\nabla u_p\|_p.
\]
Hence, $\sup_{p > p_0} \|u_p\|_\infty < \infty$. 
Due to Ascoli-Arzel\`a theorem, there exists a sequence $p_h \to \infty$ such that $\{u_{p_h}\}_h$ converges to a continuous function $u_\infty$ on $M$ uniformly, as $h \to \infty$. 
Moreover, by \eqref{eq:uniform Holder}, $u_\infty$ becomes a Lipschitz function.
In particular, for $q > p_0$, $u_{p_h}$ converges to $u_\infty$ $L^q$-strongly and $W^{1,q}$-weakly, as $h \to \infty$. 
Furthermore, since $\lim_{h \to \infty} \|u_{p_h}\|_q = \|u_\infty\|_q$, we have $\|u_\infty\|_\infty = 1$. 
Moreover, since $\nabla u_{p_h}$ converges to $\nabla u_\infty$ $L^q$-weakly, we obtain 
\[
\|\nabla u_\infty\|_q \le \liminf_{h \to \infty} \|\nabla u_{p_h}\|_q \le \liminf_{h \to \infty} \|\nabla u_{p_h}\|_{p_h}. 
\]
Here, the last inequality follows from the H\"older inequality. 
So, letting $q \to \infty$, we finally obtain 
\begin{equation} \label{eq:02}
\|\nabla u_\infty\|_\infty \le \lambda_{1,p}(A)^{1/p} + \epsilon < \infty. 
\end{equation}
By \eqref{eq:Lip}, we have $\Lip(u_\infty) = \|\nabla u_\infty\|_\infty$. 
Since $\|u_\infty\|_\infty = 1$, 
by Lemma \ref{lem:open-Lip}, we have 
\[
\mathrm{inrad}(A)^{-1} \le \Lip(u_\infty) \le \liminf_{p \to \infty} \lambda_{1,p}^D(A)^{1/p}. 
\]
This completes the proof of the lim-inf inequality of Lemma \ref{lem:JLM}.

\begin{lemma} \label{lem:inrad vs pack}
Let $\{\Omega_j\}_{0 \le j \le k}$ be a disjoint family of non-empty open subsets of a closed Riemannian manifold $M$. 
Then, we have 
\[
\min_{0 \le j \le k} \mathrm{inrad}(\Omega_j) \le \mathrm{pack}_{k+1}(M). 
\]
\end{lemma}
\begin{proof}
For $0 \le j \le k$, we set $r_j = \mathrm{inrad}(\Omega_j)$. 
Let us take $x_j \in \Omega_j$ with $|x_j, \partial \Omega_j| = r_j$ for all $0 \le j \le k$. 
Then, we have $U_{r_j}(x_j) \subset \Omega_j$. 
Hence, we obtain 
\begin{align*} \label{eq:inrad vs pack}
\min_{0 \le j \le k} r_j &\le 
\min_{0 \le i < j \le k} \frac{r_i+r_j}{2}
\le \min_{0 \le i < j \le k} \frac{|x_i,x_j|}{2} \le \mathrm{pack}_{k+1}(M).
\end{align*}
This completes the proof. 
\end{proof}

Let us give a proof of Theorem \ref{thm:1}. 
\begin{proof}[Proof of Theorem \ref{thm:1}]
Let us prove 
\begin{equation} \label{eq:limsup}
\limsup_{p \to \infty} \overline \lambda_{k,p}(M)^{1/p} \le \mathrm{pack}_{k+1}(M)^{-1}. 
\end{equation}
Let $r = \mathrm{pack}_{k+1}(M)$. 
Let us take $(x_0, x_1, \dots, x_k)$ a $(k+1)$-packer of $M$, that is, $\min_{0 \le i < j \le k} |x_i,x_j| =2 r$. 
We consider $1$-Lipschitz functions defined as
\[
u_i := \max \{ r - |x_i, \,\cdot\,|, 0 \}
\]
for $0 \le i \le k$.
Then, we have 
\begin{align*}
\lambda_{1,p}^D(U_r(x_i))^{1/p} 
\le 
\left( \frac{1}{m(U_r(x_i))} \int_{U_r(x_i)} u_i^p\, dm \right)^{-1/p}. 
\end{align*}
Letting $p \to \infty$, we obtain 
\[
\limsup_{p \to \infty} \overline \lambda_{k,p}(M)^{1/p} \le \max_{0 \le i \le k} (\|u_i\|_{L^\infty}^{-1}) = r^{-1} = \mathrm{pack}_{k+1}(M)^{-1}. 
\]
Thus, we have proved \eqref{eq:limsup}.

Let us prove
\begin{equation} \label{eq:liminf}
\liminf_{p \to \infty} \overline \lambda_{k,p}(M)^{1/p} \ge \mathrm{pack}_{k+1}(M)^{-1}. 
\end{equation}
For $\epsilon > 0$, let us take disjoint open subsets $A_{0,p}, \dots, A_{k,p} \subset M$ satisfying 
\[
\overline \lambda_{k,p}(M)^{1/p} + \epsilon \ge \max_{0 \le i\le k} \lambda_{1,p}^D(A_{i,p})^{1/p}. 
\]
We take $u_{i,p} \in W^{1,p}_0(A_{i,p})$ with 
\[
\lambda_{1,p}^D(A_{i,p})^{1/p} \le \frac{\|\nabla u_{i,p}\|_p}{\|u_{i,p}\|_p} \le \lambda_{1,p}^D(A_{i,p})^{1/p} + \epsilon. 
\]
We assume $\|u_{i,p}\|_p=1$. 
For $0 \le i \le k$ and $q \ge 2$, using an argument similar to the argument from \eqref{eq:01} to \eqref{eq:02}, 
we obtain a subsequence $p_h \to \infty$ such that $\{u_{i,p_h}\}_{h=1}^\infty$ 
converges in $L^q$-strongly and $W^{1,q}$-weakly to some $u_i$, as $h \to \infty$. 
Then we have $\|u_i\|_\infty = \lim_{h \to \infty} \|u_{i,p_h}\|_{p_h} = 1$ and 
\begin{align*}
\|\nabla u_i\|_\infty 
&\le \liminf_{p \to \infty} \lambda_{k,p}(M)^{1/p} + 2\epsilon \le \mathrm{pack}_{k+1}(M)^{-1} + 2\epsilon. 
\end{align*}
Moreover, $u_i$ is Lipschitz and $\|\nabla u_i\|_\infty = \Lip(u_i)$ by \eqref{eq:Lip}. 
Since $u_{i,p} u_{j,p} = 0$ on $M$ for $i \ne j$, we have $u_i u_j = 0 \text{ on } M$. 
Hence, $A_i := \{u_i \ne 0\}$ are mutually disjoint open subsets. 
Note that $\|u_i\|_\infty = 1$ implies $A_i \ne \emptyset$. 
By using an argument in the proof of Lemma \ref{lem:JLM} for $u_i$, we obtain 
\begin{align*}
\mathrm{inrad}(A_i)^{-1} &\le \Lip(u_i) = \|\nabla u_i\|_\infty \\
&\le \liminf_{p \to \infty} \|\nabla u_{i,p}\|_p = \liminf_{p \to \infty} \lambda_{1,p}^D(A_{i,p})^{1/p}+2\epsilon.
\end{align*}
Therefore, due to Lemma \ref{lem:inrad vs pack}, we have \eqref{eq:liminf}. 
This completes the proof of Theorem \ref{thm:1}.
\end{proof}

\subsection{Another variant} \label{sec:other lambda}
Let us define a value $\underline \lambda_{k,p}(M)$ by 
\begin{equation} \label{eq:under,k,p}
\underline \lambda_{k,p}(M) := \frac{1}{k+1} \inf_{A_0, \dots, A_k} \sum_{i=0}^k \lambda_{1,p}^D(A_i). 
\end{equation}
Here, the infimum runs over all families $\{A_i\}_{0 \le i \le k}$ of disjoint open subsets of $M$.
Obviously, $\underline {\lambda}_{k,p} \le \overline{\lambda}_{k,p}$ holds. 

\begin{corollary} \label{cor:1}
Let $M$ be a closed Riemannian manifold. 
Then, we have 
\[
\lim_{p \to \infty} \underline \lambda_{k,p}(M)^{1/p} = \mathrm{pack}_{k+1}(M)^{-1}.
\]
\end{corollary}
\begin{proof}
It is clear that $\limsup_{p \to \infty} \underline \lambda_{k,p}(M)^{1/p} \le \lim_{p \to \infty} \overline \lambda_{k,p}(M)^{1/p} = \mathrm{pack}_{k+1}(M)^{-1}$. 
Suppose that 
\[
\liminf_{p \to \infty} \underline \lambda_{k,p}(M)^{1/p} < \mathrm{pack}_{k+1}(M)^{-1} - \epsilon
\]
holds for some $\epsilon > 0$. 
Hence, there exist infinitely many $p > 1$ such that we have 
\begin{equation} \label{eq:infinite-p}
\left( \frac{1}{k+1} \inf_{A_0, \dots, A_k} \sum_{i=0}^{k} \lambda_{1,p}^D(A_i) \right)^{1/p} 
< \mathrm{pack}_{k+1}(M)^{-1} - \epsilon. 
\end{equation}
Moreover, the set $J \subset (1, \infty)$ of all $p$ satisfying \eqref{eq:infinite-p} is unbounded.
In particular, for each $p \in J$, 
there exists a disjoint family $\{B_i\}_{i=0}^k$ of open subsets of $M$ such that 
\begin{equation*} \label{eq:X_p}
\left( \frac{1}{k+1} \sum_{i=0}^{k} \lambda_{1,p}^D(B_i) \right)^{1/p}
\le 
\left( \frac{1}{k+1} \inf_{A_0, \dots, A_k} \sum_{i=0}^{k} \lambda_{1,p}^D(A_i) \right)^{1/p} + \epsilon/3. 
\end{equation*}
We denote by $X_p$ the set of all $\{B_i\}_{i=0}^k$ as above. 
Furthermore, we set 
\[
Y_p:=
\left\{ (\lambda_{1,p}^D(B_i)^{1/p} )_{i=0}^k \in \mathbb R^{k+1} \,\middle|\, 
\{B_i\}_{i=0}^k \in X_p
\right\}. 
\]
Then, we have 
\[
\sup_{p \in J} \sup_{r \in Y_p} \|r\|_p < \infty, 
\] 
where $\|r\|_p^p = \frac{1}{k+1} \sum_{i=0}^k r_i^p$ for $r = (r_i)_i \in \mathbb R^{k+1}$. 
Therefore, there exists $p_0 > 1$ such that for any $p \in J$ with $p>p_0$ and $\{A_i\}_{i=0}^k \in X_p$, we have 
\[
\max_{0 \le i \le k} \lambda_{1,p}^D(A_i) - \epsilon/3 < 
\left( 
\frac{1}{k+1} \sum_{i=0}^k \lambda_{1,p}^D(A_i)
\right)^{1/p}. 
\]
This is a contradiction.
Thus, we have the conclusion of the corollary.
\end{proof}

\subsection{Monotonicity in $k$}
We prove a monotonicity of the sequence $\left\{ \overline \lambda_{k,p} \right\}_k$ in $k$. 
\begin{proposition}
For $M = (M,m)$ as in Theorem \ref{thm:1}, we have 
\[
\overline \lambda_{k,p}(M) \le \overline \lambda_{k+1,p}(M).
\]
\end{proposition}
\begin{proof}
Let $A$ and $B$ be open subsets in $M$ with $A \cap B = \emptyset$. 
Let $\epsilon > 0$. 
Let us take $f \in W^{1,p}_0(A),$ and $g \in W_0^{1,p}(B)$ with $\|f\|_p = \|g\|_p = 1$, 
\[
\|\nabla f\|_p \le \lambda_{1,p}^D(A)^{1/p} + \epsilon \text{ and } \|\nabla g\|_p \le \lambda_{1,p}^D(B)^{1/p} + \epsilon. 
\]
Then, $u = f+g \in W^{1,p}_0(A \cup B)$.
Furthermore, we have 
\begin{align*}
\frac{\|\nabla u\|_p^p}{\|u\|_p^p} &= \frac{\int_{A\cup B} |\nabla f+\nabla g|^p\,dm}{\int_{A\cup B} |f+g|^p\,dm} \\
&= \frac{\|\nabla f\|_p^p + \|\nabla g\|_p^p}{2} 
\le \max\{ \lambda_{1,p}^D(A), \lambda_{1,p}^D(B) \}. 
\end{align*}
Therefore, we conclude that $\lambda_{1,p}^D(A \cup B) \le \max\{ \lambda_{1,p}^D(A), \lambda_{1,p}^D(B) \}$.

We take a disjoint family $\{A_i\}_{i=0}^{k+1}$ of non-empty open subsets in $M$. 
Using the final conclusion of the first paragraph, we obtain 
\[
\lambda_{1,p}^D(A_k \cup A_{k+1}) \le \max \{ \lambda_{1,p}^D(A_k), \lambda_{1,p}^D(A_{k+1}) \}.
\]
In particular, 
\begin{align*}
\overline \lambda_{k,p}(M) &\le \max_{0 \le i < k} \{ \lambda_{1,p}^D(A_i), \lambda_{1,p}^D(A_k \cup A_{k+1}) \} \\
&\le \max_{0 \le i \le k+1} \lambda_{1,p}^D(A_i). 
\end{align*}
Since $\{A_i\}_{i=0}^{k+1}$ is arbitrary, we obtain the conclusion of the proposition. 
\end{proof}

We will see that $\overline \lambda_{k,p}(M)$ goes to infinity as $k \to \infty$ (Corollary \ref{cor:k to infinity}).  

\subsection{Min-max values for Dirichlet type problem} \label{sec:Dir}
Let $M$ be a complete Riemannian manifold. 
We consider a bounded open subset $\Omega$ of $M$. 
We consider three values defined as 
\begin{align*}
\inpack_k (\Omega) &:= 
\max_{x_1, \dots, x_k \in \Omega} \min_{1\le i \ne j \le k} \left\{ \frac{|x_i,x_j|}{2}, |x_i, \partial \Omega| \right\}, \\
\overline \lambda_{k,p}^D(\Omega) &:= \inf_{A_1, \dots, A_k} \max_{1 \le i \le k} \lambda_{1,p}^D(A_i), \\
\underline \lambda_{k,p}^D(\Omega) &:= \frac{1}{k}\inf_{A_1, \dots, A_k} \sum_{1 \le i \le k} \lambda_{1,p}^D(A_i).
\end{align*}
Here, in the last two values, the infimums run over all disjoint open subsets $A_i$ of $M$ such that $A_i \subset \Omega$ and that $\partial A_i \subset \overline{\Omega}$. 
The first value $\inpack_k (\Omega)$ is called the {\it $k$-th inscribed packing radius} of $\Omega$, which is the largest $r > 0$ such that $k$ balls of radius $r$ of $M$ are contained in $\Omega$ and are mutually disjoint.
Clearly, $\inpack_1(\Omega)$ is the usual inradius $\inrad(\Omega)$. 
Furthermore, we note that $\lambda_{1,p}^D = \overline \lambda_{1,p}^D = \underline \lambda_{1,p}^D$.
Due to an argument similar to the proof of Theorem \ref{thm:1} and Corollary \ref{cor:1}, we obtain: 
\begin{theorem} \label{thm:1-1}
For $\Omega$ as above, we have
\[
\lim_{p \to \infty} \overline \lambda_{k,p}^D(\Omega)^{1/p} = \lim_{p \to \infty} \underline \lambda_{k,p}^D(\Omega)^{1/p} = \inpack_{k}(\Omega)^{-1}.
\]
\end{theorem}

\begin{remark} \upshape
The values $\overline \lambda_{k,p}^D$, $\underline \lambda_{k,p}^D$ and $\inpack_k$ are well-defined for compact Riemannian manifold with boundary. 
A corresponding statement to Theorem \ref{thm:1-1} for compact Riemannian manifolds with boundary also holds. 
\end{remark}

\subsection{Domain monotonicities}

The following two propositions directly follows from the definition of $\overline \lambda_{k,p}^D$ and $\overline \lambda_{k,p}$: 
\begin{proposition} \label{prop:dom mon-D}
Let $U, V$ be bounded open subsets in a complete Riemannian manifold.
If $U \subset V$, then 
\[
\overline \lambda_{k,p}^D(V) \le \overline \lambda_{k,p}^D(U).
\] 
\end{proposition}

\begin{proposition} \label{prop:dom mon-M}
Let $M$ be a closed Riemannian manifold and $U$ an open subset of $M$. 
Then, we have 
\[
\overline \lambda_{k,p}(M) \le \overline \lambda_{k+1,p}^D(U).
\] 
If $U$ is the union of disjoint family of non-empty open sets $U_0, \dots, U_k$, then we have 
\[
\overline \lambda_{k+1,p}^D(U) \le \max_{0 \le i\le k} \lambda_{1,p}^D(U_i). 
\]
\end{proposition}

\begin{remark} \upshape 
In \cite{FS}, Funano and Sakurai considered higher-order $p$-Poincar\'e constants $\nu_{k,p}$ and $\nu_{k,p}^D$. 
Such values satisfy domain monotonicity formulas as in Proposition \ref{prop:dom mon-D}. 
Using this, they obtained \cite[Theorems 1.1, 3.4 and 1.2]{FS}. 
So, we also obtain statements similar to their results for $\overline \lambda_{k,p}$ and $\overline \lambda_{k,p}^D$. 
\end{remark}

\subsection{Lindqvist-Matei's type monotonicity}

The following is proved in \cite{L2} and \cite{Matei2} for the case of $k=1$: 
\begin{proposition}
Let $M$ be a closed Riemannian manifold. Then, the function $p \mapsto p \overline \lambda_{k,p}(M)^{1/p}$ is non-decreasing, for each $k$. 
If $\Omega$ is a bounded open subset of a complete Riemannian manifold, then the function $p \mapsto p \overline \lambda_{k,p}^D(\Omega)^{1/p}$ is non-decreasing, for each $k$. 
\end{proposition}
\begin{proof}
Let us recall that, in \cite{L2}, Lindqvist proved that the map $p \mapsto p\lambda_{1,p}^D(\Omega)^{1/p}$ is strictly increasing on $(1, \infty)$, for every bounded open set $\Omega$ in a complete Riemannian manifold. 
For $\epsilon > 0$, we take $A_0, \dots, A_k$ disjoint non-empty open sets in $M$ such that 
\[
\max_{0 \le i \le k} p\lambda_{1,p}^D(A_i)^{1/p} < p\overline \lambda_{k,p}(M)^{1/p} + \epsilon.
\]
Hence, for $q < p$, we have 
\[
q\overline \lambda_{k,q}(M)^{1/q} \le \max_{0 \le i \le k} q\lambda_{1,q}^D(A_i)^{1/q} < \max_{0 \le i \le k} p\lambda_{1,p}^D(A_i)^{1/p}. 
\]
Letting $\epsilon \to 0$, we have 
\[
q \overline \lambda_{k,q}(M)^{1/q} \le p \overline \lambda_{k,p}(M)^{1/p}. 
\]
A proof of the case of $\overline \lambda_{k,p}^D(\Omega)$ is similar to the proof as above. 
This completes the proof. 
\end{proof}
As in \cite{L2} and \cite{Matei2}, it might be true that $p \overline \lambda_{k,p}(M)^{1/p}$ and $p \overline{\lambda}_{k,p}^D(\Omega)$ are strictly increasing. 

\subsection{Comparison with real and fake spectra} \label{sec:real vs fake}
Let $M$ be a closed Riemannian manifold.
To the author's knowledge, there exist at least three ways to construct real spectra of $\triangle_p$. 
We denote such spectra by $\{\lambda_{k,p}^-(M)\}_k, \{ \lambda_{k,p}(M)\}_k$ and $\{ \lambda_{k,p}^+(M)\}_k$, following \cite{PAO}. 
From the left, they are introduced by Krasnoselskii (\cite{Kra}), by Perera (\cite{Perera}) and by Dr\'abek and Robinson \cite{DR}, respectively. 
For the definitions of them, we refer to \cite{PAO}. 
It is known that $\{\lambda_{k,p}^{\pm}\}_k$, $\{\lambda_{k,p}\}_k$ are real spectra of $\triangle_p$, that is, they are eigenvalues, and that are unbounded sequences. 
From the definitions we know
\[
\lambda_{k,p}^-(M) \le \lambda_{k,p}(M) \le \lambda_{k,p}^+(M).
\]
See, for instance, \cite[Proposition 4.7]{PAO}. 

Spectra $\lambda_{k,p}^{D,-}$, $\lambda_{k,p}^D$ and $\lambda_{k,p}^{D,+}$ for Dirichlet eigenvalue problem of $\triangle_p$ are also defined similar to the above case. 
These are unbounded and $\lambda_{k,p}^{D,-} \le \lambda_{k,p}^D \le \lambda_{k,p}^{D,+}$ holds. 

\begin{proposition} \label{prop:comparision}
Let $M$ be a closed Riemannian manifold and $U$ a bounded open set in a complete Riemannian manifold. 
Then we have 
\[
\lambda_{k,p}^+(M) \le \overline \lambda_{k,p}(M) \text{ and } \lambda_{k,p}^{D,+}(U) \le \overline \lambda_{k,p}^D(U). 
\]
\end{proposition}
\begin{proof}
To prove the first inequality, we compare a value $\widehat \nu_{k,p}$ considered in \cite{FS} with $\overline {\lambda}_{k,p}$. 
Here, $\widehat \nu_{k,p}$ is defined as 
\[
\widehat \nu_{k,p}(M) := \inf_{V} \sup_{u \in V \setminus \{0\}} \frac{\|\nabla u\|_p^p}{\|u\|_p^p}, 
\]
where the infimum runs over all $(k+1)$-dimensional linear subspaces $V$ of $W^{1,p}(M)$, that is called a modified $k$-th $p$-Poincar\'e constant. 
As discussed in \cite[Remark 2.1]{FS}, we can see that 
\[
\lambda_{k,p}^+ \le \widehat \nu_{k,p}.
\]
So, it suffices to show that 
\begin{equation} \label{eq:vs FS}
\widehat \nu_{k,p} \le \overline \lambda_{k,p}.
\end{equation}

Let $U_0, \dots, U_k$ be disjoint non-empty open subsets of $M$. 
Fix $\epsilon > 0$. 
Let $u_i \in W^{1,p}_0(U_i)$ satisfies $\|u_i\|_p = 1$ and $\lambda_{1,p}^D(U_i) \le \|\nabla u_i\|_p^p \le \lambda_{1,p}^D(U_i) + \epsilon$. 
Since $\{U_i\}_i$ is disjoint, $\{u_i\}_{i=0}^k$ is linearly independent. 
For $(t_i)_{i=0}^k \in \mathbb R^{k+1} \setminus \{0\}$, we have 
\[
\widehat \nu_{k,p}(M) \le 
\frac{\left\| \nabla \sum_i t_i u_i \right\|_p^p}{\left\| \sum_i t_i u_i \right\|_p^p} 
= \frac{\sum_i |t_i|^p \| \nabla u_i\|_p^p}{\sum_i |t_i|^p} \le \max_{0 \le i \le k} \|\nabla u_i\|_p^p \le \max_{0 \le i \le k} \lambda_{1,p}^D(U_i) + \epsilon.
\]
Letting $\epsilon \to 0$, we have $\widehat \nu_{k,p} \le \max_i \lambda_{1,p}^D(U_i)$. 
Since $\{U_i\}_i$ is an arbitrary disjoint family of open sets, we obtain \eqref{eq:vs FS}. 
This completes the proof of the first statement.
The second one is proved by an argument similar to the proof as above. 
\end{proof}

As a corollary to the above proposition, we have 
\begin{corollary} \label{cor:k to infinity}
$\lim_{k \to \infty} \overline \lambda_{k,p}(M) = \infty$ and $\lim_{k \to \infty} \overline \lambda_{k,p}^D(U) = \infty$. 
\end{corollary}

\section{Generalization to singular spaces} \label{sec:gen}
In this section, we give a generalization of Theorem \ref{thm:1} and Corollary \ref{cor:1} to singular metric measure spaces. 
Note that the proofs of all the results above do not rely on the assumption that $m$ is the Riemannian measure of $(M,g)$. 
Moreover, the underlying topology of $M$ is not so important. 
We have only needed to care whether $(M,m)$ satisfies \eqref{eq:uniform Holder} and \eqref{eq:Lip} or not.
So, we explain a sufficient condition for (replacements of) \eqref{eq:uniform Holder} and \eqref{eq:Lip} being valid.
Purposes of this section are to explain the precise meaning of the following theorem, to give a proof of it and to show examples of spaces satisfying the assumption of the theorem: 
\begin{theorem} \label{thm:2}
Let $(X,m)$ be a  compact metric measure space, where $X$ is geodesic and $m$ is a Borel probability measure on $X$ with full support. 
We assume that $(X,m)$ is doubling \eqref{eq:doubling}, supports the Poincar\'e inequality \eqref{eq:Poincare} and has the Sobolev-to-Lipschitz property \eqref{eq:Lip3} described below.
Then, we have 
\[
\lim_{p \to \infty} \overline \lambda_{k,p}(X,m) = \lim_{p \to \infty} \underline \lambda_{k,p}(X,m) = \pack_{k+1}(X)^{-1}
\]
for $k \ge 1$. 
\end{theorem}
This is a generalization of corresponding statements in \cite[Theorem 1.1]{Ho} and \cite[Theorem 1.9.6]{AH}, to arbitrary $k$.

In the following, first we give the definition of the conditions \eqref{eq:doubling} and \eqref{eq:Poincare}. 
Next, we give the definition of the condition \eqref{eq:Lip3}. 

\subsection{Basic conditions}
We now explain the assumptions \eqref{eq:doubling}
and \eqref{eq:Poincare} mentioned in Theorem \ref{thm:2} 
which 
imply a numerical Morrey-type inequality \eqref{eq:Morrey} with a uniform estimate \eqref{eq:uniform Holder}.
Remark that
several conditions are simpler than corresponding ones given in literatures. 

Let $X$ denote a 
complete separable 
metric space. 
We say that $X$ is {\it geodesic} if for any $x,y \in X$, there exists a continuous curve $\sigma : [0,1] \to X$ such that 
$\sigma(0)=x$, $\sigma(1)=y$ and $|\sigma(s), \sigma(t)| = |x,y||s-t|$ for $s, t \in [0,1]$. 
Such a curve $\sigma$ is called a minimal geodesic from $x$ to $y$.
Let $m$ be a Borel probability measure on $X$ with full support. 
We call such a pair $(X,m)$ a metric measure space. 
We often call $(X,m)$ a geodesic metric measure space, to emphasize that $X$ is geodesic. 
When $X$ is compact, then we call $(X,m)$ a compact (geodesic) metric measure space. 

\begin{definition} \upshape
A  metric measure space $(X,m)$ is said to be {\it doubling} if 
there exists $C_D > 0$ such that 
\begin{equation} \label{eq:doubling}
m(U_{2r}(x)) \le C_D m(U_r(x))
\end{equation}
holds, for every $x \in X$ and $r > 0$. 
\end{definition}

For $f \in \Lip(X)$ and $x \in X$, the {\it local Lipschitz constant} of $f$ at $x$ is defined as 
\[
\lip(f)(x) = \lim_{r \to 0+} \sup_{y \in U_r(x) \setminus \{x\}} \frac{|f(x)-f(y)|}{|x,y|}.
\]
\begin{definition} \upshape
We say that $(X,m)$ {\it supports the Poincar\'e inequality} if there exists $p_0 \ge 1$, $C_P > 0$ and $\sigma \ge 1$ such that 
\begin{equation} \label{eq:Poincare}
\dashint_{U_r(x)} \left| 
f - \dashint_{U_r(x)} f\,dm 
\right|dm 
\le C_P \left(\dashint_{U_{\sigma r}(x)} \lip (f)^{p_0}\, dm\right)^{1/p_0}
\end{equation}
holds for $x \in X$, $r > 0$ and $f \in \Lip(X)$. 
Here, we use the standard notation: 
\[
\dashint_{U_r(x)} f \,dm = \frac{1}{m(U_r(x))} \int_{U_r(x)} f\,dm.
\]
\end{definition}
The inequality \eqref{eq:Poincare} is called a $p_0$-Poincar\'e inequality. 

Let us assume that $(X,m)$ is doubling \eqref{eq:doubling} and supports the Poincar\'e inequality \eqref{eq:Poincare}. 
Then the well-behaved $(1,p)$-Sobolev space is defined (\cite{Ch}, \cite{Ha}, \cite{Sha}) and is known that $\Lip(X)$ is dense in there. 
Indeed, 
for $f \in L^2(X)$, its {\it minimal relaxed gradient} $|Df|_\ast$ is defined as a nonnegative Borel function on $X$ satisfying a particular minimality condition (see \cite[Definition 4.2, Lemma 4.3]{AGS0}). 
This is a counterpart of the absolute gradient in the non-smooth setting.
When $f \in L^2(X) \cap \Lip(X)$, it is known that 
\begin{equation} \label{eq:Df=lip f}
|Df|_\ast (x) = \lip(f)(x)
\end{equation}
holds for $m$-almost everywhere, because we suppose \eqref{eq:doubling} and \eqref{eq:Poincare} (\cite[Theorem 5.1]{Ch}, \cite[Theorem 6.2]{AGS0}). 
Using the minimal relaxed gradient, we define the $(1,p)$-norm of $f \in L^p(X)$ as 
\[
\|f\|_{1,p} := \left( \|f\|_p^p + \||Df|_\ast \|_p^p \right)^{1/p}. 
\]
For $1 < p \le \infty$, the $(1,p)$-Sobolev space $W^{1,p}(X,m)$ is defined as the subspace of $f \in L^p(X)$ consisting of elements with finite $(1,p)$-norm. 
It is known that for $f \in W^{1,p}(X,m)$, the pair $(f, |Df|_\ast)$ satisfies \eqref{eq:Poincare} instead of the pair $(f, \lip(f))$ (see \cite[Proposition 8.1.3]{HKST}, \cite[Theorem 3.6]{AGS0}). 

We will prove the next theorem in Appendix \ref{sec:Morrey}
\begin{theorem}[{\cite[(25) Theorem 5.1]{HK}}] \label{thm:HK}
Let $(X,m)$ be a compact metric measure space being doubling \eqref{eq:doubling} and supporting the Poincar\'e inequality \eqref{eq:Poincare}. 
Then, for $p > \max\{p_0,s\}$ and for $f \in W^{1,p}(X,m)$, 
\[
|f(x)-f(y)| \le C(p) \||Df|_\ast\|_p |x,y|^{1-s/p}
\]
holds for every $x,y \in X$, where $s = \log_2 C_D$ and $p_0$ is the exponent appeared in \eqref{eq:Poincare}. 
Here, $C(p)$ is a constant depending only on $C_D$, $C_P$ appeared in \eqref{eq:doubling} and \eqref{eq:Poincare}, $d = \mathrm{diam}(X)$ and $p$ such that $\sup_{p > \max \{p_0,s\}} C(p) < \infty$. 
\end{theorem}
So, this is a counterpart of \eqref{eq:uniform Holder} in the metric-measure setting. 

Furthermore, $\lambda_{1,p}^D(\Omega)$ is defined as 
\begin{equation} \label{eq:Dir-1,p}
\lambda_{1,p}^D(\Omega) = \inf
\left\{ 
\frac{\|\lip(f)\|_p^p}{\|f\|_p^p} \,\middle|\, f \in \Lip_0(\Omega)
\right\}
\end{equation}
for an open set $\Omega$ in $X$, where $\Lip_0(\Omega)$ stands for the set of all Lipschitz functions with compact support within $\Omega$. 
Note that \eqref{eq:Dir-1,p} coincides with \eqref{eq:1st-D} in the smooth setting.
So, using \eqref{eq:Dir-1,p}, we define $\overline \lambda_{k,p}(X,m)$ and $\underline \lambda_{k,p}(X,m)$ by the same formulas as \eqref{eq:over,k,p} and \eqref{eq:under,k,p}, respectively. 

\subsection{A counterpart of \eqref{eq:Lip}}
In the last subsection, we obtain a condition to support a counterpart of \eqref{eq:uniform Holder}. 
Next, we give a condition which support \eqref{eq:Lip}. 
Let $(X,m)$ be a metric measure space. 
\begin{definition}[{\cite[Definition 4.9]{Gi}}] \upshape \label{def:Sob_to_Lip}
We say that $(X,m)$ {\it has the Sobolev-to-Lipschitz property} if for any $f \in W^{1,2}(X,m) \cap W^{1,\infty}(X,m)$, then $f$ has a Lipschitz representative $\tilde f$ such that 
\begin{equation} \label{eq:Lip3}
\Lip(\tilde f) = \||D f|_\ast \|_\infty.
\end{equation}
\end{definition}
This is a direct assumption for \eqref{eq:Lip} being valid. 

As written after \cite[Definition 4.9]{Gi}, 
we show examples which have the Sobolev-to-Lipschitz property \eqref{eq:Lip3}: 
\begin{itemize}
\item CD$(K,N)$-spaces (\cite{Raj} and see a discussion in the after \cite[Definition 4.9]{Gi}). Here, $K \in \mathbb R$ and $1 < N < \infty$. 
\item RCD$(K,\infty)$-spaces (\cite{AGS1}). Here, $K \in \mathbb R$. 
\end{itemize}
For the definitions of CD-spaces and RCD-spaces, we refer to \cite{LV}, \cite{St}, \cite{AGS1}. 
Such spaces are regarded as generalized objects of Riemannian manifolds with weighted Ricci lower curvature bound and with an upper bound of the dimension, in a synthetic sense. 
Note that a compact CD$(K,N)$-space automatically satisfies
\eqref{eq:doubling}, and is known to satisfy \eqref{eq:Poincare} (\cite[Theorem 2]{Raj}).
Furthermore, if a compact CD$(K,\infty)$-space is doubling, then it supports the Poincar\'e inequality \eqref{eq:Poincare} (\cite[Theorem 1]{Raj}). 

We consider a more condition. 
For a nonnegative Borel function $f$ on $X$ and $x,y \in X$, we set 
\[
\mathcal F_f(x,y) := \inf_{\gamma} \int_0^{|x,y|} f\circ \gamma(s)\,ds
\]
where the infimum runs over all minimal geodesic $\gamma$ from $x$ to $y$. 
Here, we recall that $X$ is assumed to be geodesic. 

\begin{definition}[\cite{CC3}] \label{def:segment} \upshape 
We say that $(X,m)$ {\it satisfies the segment inequality} if 
there exists $C_S>0$ such that 
\begin{equation} \label{eq:segment} 
\int_{U_r(x) \times U_r(x)} \mathcal F_f(y,z) \, d(m \otimes m) \le C_S r m(U_r(x)) \int_{U_{3r}(x)} f \,dm
\end{equation}
holds for every $x \in X$, every $r > 0$ and every nonnegative Borel function $f$ on $X$. 
\end{definition}
Such a condition was appeared in \cite[Theorem 2.11]{CC}.
We employ a formulation in \cite{Ho}. 

We now summarize a relation among above conditions. 
\begin{proposition} \label{prop:Lip}
Let $(X,m)$ be a compact geodesic metric measure space 
being doubling \eqref{eq:doubling}. 
We consider the following conditions. 
\begin{enumerate}
\item $(X,m)$ satisfies the segment inequality \eqref{eq:segment}. 
\item $(X,m)$ has the Sobolev-to-Lipschitz property (Definition \ref{def:Sob_to_Lip}). 
\item $(X,m)$ satisfies \eqref{eq:Lip}, that is, for every $f \in \Lip(X)$, we have $\|\lip(f)\|_\infty = \Lip(f)$. 
\end{enumerate}
Then, $(2)$ implies $(3)$ and $(1)$ implies $(3)$.
Moreover, if $(X,m)$ supports the Poincar\'e inequality \eqref{eq:Poincare}, then $(3)$ and $(2)$ are equivalent.
\end{proposition}
\begin{proof}
The implication (1) $\Rightarrow$ (3) is proved by \cite[Proposition 2.8]{Ho}. 
The implication (2) $\Rightarrow$ (3) is trivial, because of \eqref{eq:Df=lip f}. 
Supposing that $(X,m)$ supports the Poincar\'e inequality, we prove (3) $\Rightarrow$ (2). 
Let us take $f \in W^{1,\infty}(X,m)$. 
Since $\| |Df|_\ast \|_p \le \| |Df|_\ast \|_\infty$, by Theorem \ref{thm:HK}, if $p$ is large enough, 
then $f$ has a H\"older representative. 
We denote it by $f$ again. 
Since $f$ is $(1-O(p))$-H\"older and its H\"older constant is uniformly bounded, $f$ is Lipschitz. 
So, by (3) and \eqref{eq:Df=lip f}, we have $\||Df|_\ast \|_\infty = \|\lip(f)\|_\infty = \Lip(f)$. 
This completes the proof. 
\end{proof}

Let us start a proof of Theorem \ref{thm:2}. 
\begin{proof}[Proof of Theorem \ref{thm:2}]
Let $(X,m)$ be as in Theorem \ref{thm:2}. 
Due to Theorem \ref{thm:HK} and Proposition \ref{prop:Lip}, the argument of Theorem \ref{thm:1} works in this setting. 
So, we obtain the conclusion. 
\end{proof}

As mentioned after Definition \ref{def:Sob_to_Lip}, compact CD$(K,N)$-spaces and compact RCD$(K,\infty)$-spaces having doubling measure satisfy the assumption of Theorem \ref{thm:2}. 

\subsection{An example satisfying the assumption of Theorem \ref{thm:2}, but is not CD$(K,\infty)$} \label{sec:example}
In this subsection, we show an example that satisfies the assumption of Theorem \ref{thm:2}, but it is not a CD$(K,\infty)$-space for every $K \in \mathbb R$. 
In particular, it is neither an RCD$(K,\infty)$-space nor a CD$(K,N)$-space, for $K \in \mathbb R$ and $1 < N < \infty$.

\subsubsection{Basic definitions}
For a proof of the following facts, we refer to \cite{Vi} and \cite{LV}. 
Let $X = (X,d)$ be a compact metric space. 
Let us denote by $P(X)$ the set of all Borel probability measures on $X$. 
For $\mu, \nu \in P(X)$, their {\it coupling} $\xi$ is a measure $\xi \in P(X \times X)$ satisfying $\xi(A \times X) = \mu(A)$ and $\xi(X \times A) = \nu(A)$ for every Borel set $A \subset X$. 
The $L^2$-{\it Wasserstein distance} $W(\mu, \nu)$ between $\mu$ and $\nu$ is defined by 
\[
W(\mu,\nu) := \inf_\xi \|d\|_{L^2(X \times X, \xi)}, 
\]
where the infimum runs over all couplings $\xi$ between $\mu$ and $\nu$. 
It is known that a minimizer $\xi$ of $W(\mu,\nu)$ exists. 
Such a coupling is called an {\it optimal coupling} (with respect to $W$). 
We call the pair $(P(X), W)$ the {\it $L^2$-Wasserstein space} over $X$. 
Since $X$ is compact, so is $P(X)$. 
Moreover, if $X$ is geodesic, so is $P(X)$. 
A geodesic in $(P(X), W)$ is called a {\it Wasserstein geodesic}. 

Let us explain a relation among Wasserstein geodesics, optimal couplings and optimal transference plans. 
Let us assume that $X$ is compact and geodesic. 
Let us denote by $\mathrm{Geo}(X)$ the set of all minimal geodesics in $X$ parametrized by $[0,1]$. 
For $t \in [0,1]$, we define the evaluation of curves at $t$ as
\[
e_t : C([0,1],X) \ni \gamma \mapsto \gamma(t) \in X. 
\]
Here, $C([0,1],X)$ stands for the set of all continuous curves from $[0,1]$ to $X$ equipped with the uniform topology. 
This map implies the push-forward of measures: 
\[
(e_t)_\# : P(C([0,1],X)) \ni \pi \mapsto (e_t)_\# \pi \in P(X). 
\]
It is known that for $\mu,\nu \in P(X)$ and an optimal coupling $\xi$ of them, there exists $\pi \in P(\mathrm{Geo}(X))$ such that $[0,1] \ni t \mapsto (e_t)_\# \pi \in P(X)$ is a Wasserstein geodesic from $\mu$ to $\nu$ and that $\xi = (e_0, e_1)_\# \pi$. 
Conversely, every Wasserstein geodesic is obtained as above (see \cite[Corollary 7.22]{Vi}, \cite[Proposition 2.10]{LV}). 
Such a $\pi$ is called an {\it optimal transference plan} from $\mu$ to $\nu$. 


We now recall the following theorem: 
\begin{theorem}[{\cite[Corollary 1.4]{RajS}}] \label{thm:RS}
Let $(X,d,m)$ be as above. 
If $(X,d,m)$ is strong CD$(K,\infty)$ for some $K \in \mathbb R$, then 
for every $\mu, \nu \in P(X)$ which are absolutely continuous in $m$, there exists a unique optimal transference plan from $\mu$ to $\nu$. 
\end{theorem}
In the above situation, a Wasserstein geodesic from $\mu$ to $\nu$ is unique, due to the correspondence between optimal transference plans and Wasserstein geodesics. 
For the definition of {\it strong CD-condition}, we refer to \cite{RajS}.
Moreover, the optimal transference plan given in Theorem \ref{thm:RS} is induced by a map (for precise meaning, we refer to \cite{RajS}).

Now, we prove: 
\begin{proposition} \label{prop:example}
There exists a compact geodesic metric measure space $X$ such that $X$ is doubling, supports the Poincar\'e inequality and satisfies the Sobolev-to-Lipschitz property.
However, it is not a CD$(K,\infty)$-space for any $K \in \mathbb R$.
\end{proposition}
The desired space $X$ in Theorem \ref{prop:example} is the space considered in \cite[Example 2.9]{LV}. 
Let us explain this space. 
Let $A, B$ and $C$ be given as subsets of the plane as 
\begin{align*}
A &:= \{(x_1,0) \mid -2 \le x_1 \le -1 \}, \\
B &:= \{(x_1,x_2) \mid x_1^2+x_2^2 = 1 \}, \\
C &:= \{(x_1,0) \mid 1 \le x_1 \le 2 \}. 
\end{align*}
Then, $X$ is realized as the union of $A,B$ and $C$. 
Here, the distance function is given as intrinsic one. 
We consider a measure on $X$ which is the standard one-dimensional Hausdorff measure $H^1$. 

\begin{proof}
Let $X=A \cup B \cup C$ be as above. 
Let $v_- := (-1,0)$ and $v_+ := (1,0)$. 
It can be directly checked that $X$ is doubling (or by a general result in \cite{Pau}). 
Furthermore, using \cite[Theorem 6.15]{HeiKo} twice, $X$ is known to support the Poincar\'e inequality.

We consider $Y = A \cup B$.  
We show that $Y$ has the Sobolev-to-Lipschitz property. 
Let us take $f \in \Lip(Y)$ and consider $f|_A  \in \Lip(A)$ and $f|_B \in \Lip(B)$. 
Since both $A$ and $B$ are one-dimensional manifolds, 
\[
\| \lip(f|_D) \|_{L^\infty(D, H^1)} = \Lip(f|_D)
\]
holds for $D=A,B$. 
For $x \in A \setminus \{v_-\}$ and $y \in B \setminus \{v_-\}$, we have 
\begin{align*}
\frac{|f(x)-f(y)|}{|x,y|} &\le \frac{|f(x)-f(v_-)|+|f(y)-f(v_-)|}{|x,v_-|+|y,v_-|} \\
&\le \frac{|x,v_-|}{|x,v_-|+|y,v_-|} \frac{|f(x)-f(v_-)|}{|x,v_-|} \\
&\hspace{1.4em}+ \frac{|y,v_-|}{|x,v_-|+|y,v_-|} \frac{|f(y)-f(v_-)|}{|y,v_-|} \\
&\le \max \{ \Lip(f|_A), \Lip(f|_B)\} \\
&\le \Lip(f). 
\end{align*}
Therefore, we conclude 
\[
\max \{ \Lip(f|_A), \Lip(f|_B) \} = \Lip(f).
\]
Hence, we obtain 
\begin{align*}
\| \lip(f) \|_{L^\infty(Y,H^1)} &= \max \{\| \lip(f|_A) \|_{L^\infty(A, H^1)}, \| \lip(f|_B) \|_{L^\infty(B, H^1)} \} \\
&= \max \{ \Lip(f|_A), \Lip(f|_B) \} = \Lip(f).
\end{align*}
Therefore, by Proposition \ref{prop:Lip}, we know that $Y$ has the Sobolev-to-Lipschitz property.
Applying this argument to $(Y,C,v_+)$ instead of $(A,B,v_-)$, we conclude that $X$ has the Sobolev-to-Lipschitz property. 

We prove that $X$ is not a CD$(K,\infty)$-space for every $K \in \mathbb R$. 
Suppose that $X$ is CD$(K,\infty)$ for some $K$. 
It is trivial that $X$ is locally CAT$(0)$ (which means that $X$ has nonpositive sectional curvature in Alexandrov sense. see \cite{BBI}, \cite{MGPS} for the definition).  
So, by \cite{MGPS}, $X$ is known to be infinitesimally Hilbertian (see \cite{MGPS}, \cite{AGS1} for the definition). 
Therefore, 
$X$ is an RCD$(K,\infty)$-space. 
In particular, $X$ is a strong CD$(K,\infty)$-space. 
Let us consider two measures $\mu_0, \mu_1 \in P(X)$ defined as the uniform measures on $A$ and $C$, respectively. 
Then, as mentioned in \cite[Example 2.9]{LV}, there are uncountably many Wasserstein geodesics from $\mu_0$ to $\mu_1$. 
This contradicts to the uniqueness of optimal transference plan (Theorem \ref{thm:RS}). 
This completes the proof.
\end{proof}

%
%
\section{A note on asymptotic law for packing radii} \label{sec:Weyl}
Let $M$ be a closed $n$-dimensional Riemannian manifold. 
Due to Gromov \cite[\S 2]{Grom}, the asymptotic packing equality 
\begin{equation} \label{eq:Gromov}
\lim_{k \to \infty} \frac{ k\, \pack_k(M)^n}{\mathrm{vol}_g(M)} = \circledcirc_n 
\end{equation}
holds, where $\circledcirc_n$ is a universal constant independent on $M$, that is the Euclidean packing constant. 
Moreover, $\omega_n \circledcirc_n$ is the optimal density of sphere packings of Euclidean space $\mathbb R^n$, where $\omega_n$ is the volume of the unit ball in $\mathbb R^n$. 

By \eqref{eq:Gromov}, we immediately obtain: 
\begin{proposition} \label{prop:pack-law}
For $M$ as above, we have 
\[
\lim_{r \to 0} r^n \# \{ k \ge 1 \mid \mathrm{pack}_{k+1}(M) > r \} = \circledcirc_n \mathrm{vol}_g(M). 
\]
\end{proposition}
\begin{proof}
Let us recall the following well-known fact. 
Let $\{a_k\}_{k =1}^\infty$ be a monotone non-increasing sequence of positive numbers converging to zero. 
Suppose that $\{ k a_k\}_k$ has the limit as $k \to \infty$. 
Define a counting function as 
\[
N(r) := \# \{k \mid a_k > r\} = \max \{k \mid a_k > r\} = \min \{k \mid a_k \le r \}-1.
\]
Then, $\{r N(r)\}_{r > 0}$ has the limit as $r \to 0$ and 
\[
\lim_{r \to 0} r N(r) = \lim_{k \to \infty} k a_k.
\]
Applying this fact to $a_k = \pack_k(M)^n$, we obtain the conclusion. 
\end{proof}

Let $\lambda_{k,p}(M)$ denote a $k$-th eigenvalue of the $p$-Laplacian considered in Section \ref{sec:real vs fake}. 
Recently, in \cite{Maz}, Mazurowski generalized the classical Weyl's asymptotic law for $\{\lambda_{k,p}\}_k$: 
\begin{theorem}[\cite{Maz}] \label{thm:Maz-law}
There exists a universal constant $c_n(p)$ depending only on $n$ and $p$ such that 
\begin{equation} \label{eq:Maz-law}
\lim_{\xi \to \infty} \frac{ \# \{k \ge 1 \mid \lambda_{k,p}(M)^{1/p} < \xi \}}{\xi^n} = c_n(p) \mathrm{vol}_g(M)
\end{equation}
holds. 
\end{theorem}

Concerning with Theorem \ref{thm:Maz-law} and Proposition \ref{prop:pack-law}, we formulate the following conjecture:
\begin{conjecture} \label{conj:c(p)}
$c_n(p)$ in Theorem \ref{thm:Maz-law} is continuous in $p$ and $\lim_{p \to \infty} c_n(p) = \circledcirc_n$.
\end{conjecture}

\appendix
\section{A uniform Morrey type estimate} \label{sec:Morrey}
We prove Theorem \ref{thm:HK} following \cite{HK} and \cite{HKST}.
After that, we verify \eqref{eq:uniform Holder}. 

Let $X$ be a compact metric space and $m$ is a finite Borel measure on $X$ with full support. 
We do not assume that neither $X$ is geodesic nor $m(X)=1$.
However, we suppose that $(X,m)$ is doubling \eqref{eq:doubling} 
and denote by $C_D$ a constant appeared in \eqref{eq:doubling}. 
Then, we have 
\begin{equation} \label{eq:volume comp}
\frac{m(U_r(x))}{m(U_{r'}(x'))} \le 2 C_D \left(\frac{r}{r'} \right)^s \text{ and } \frac{m(U_r(x))}{m(U_{r'}(x))} \le C_D \left(\frac{r}{r'} \right)^s
\end{equation}
for every $x \in X$, $x' \in U_r(x)$ and $0<r'<r$. 
Here, $s = \log_2 C_D$.
For a proof, see \cite[Lemma 8.1.13]{HKST}.

As in \cite[p.25]{HK}, 
for $\sigma \ge 1$, $p >0$, an open set $\Omega$ and a Borel function $h : X \to \mathbb R$, we define a generalized Riesz potential by 
\[
J_{p}^{\sigma,\Omega} h(x) := \sum_{i \in \mathbb Z; 2^i \le 2 \sigma \mathrm{diam}(\Omega)} 2^i \left( \dashint_{B_i(x)} |h|^{p}\,dm \right)^{1/p}, 
\]
where, $B_i(x) := U_{2^i}(x)$. 
About this operation, the following is known in \cite{HK}, but we give a proof, because we want to know an explicit bound of constants appeared there. 
\begin{theorem}[{\cite[Theorems 5.2, 5.3]{HK}}] \label{thm:HK2}
Let $(X,m)$ be as above. 
Let $(f,g)$ be a pair supporting the $p$-Poincar\'e inequality in the sense that 
\[
\dashint_{U_r(x)} \left| f - \dashint_{U_r(x)} f\,dm \right|\,dm \le C_P \left( \dashint_{U_{\sigma r}(x)} |g|^{p}\,dm \right)^{1/p}
\]
holds for every $x \in X$ and $r > 0$, and let $h \in L^p(X)$. 
Here, $C_P > 0, p > 0$ and $\sigma \ge 1$ be constants. 
Then, the following holds. 
\begin{enumerate}
\item For $x \in X$, $r > 0$, we have 
\[
\left| f(y) - \dashint_{U_r(x)} f\, dm \right| \le C (J_{p}^{\sigma, U_r(x)} g(y))
\]
for almost everywhere in $U_r(x)$, where 
\[
C = (1+C_D)C_DC_P \sigma^{-1}.
\]
\item 
If $p>s$, then for every $x \in X$ and $r > 0$, we have 
\[
\| J_{p}^{\sigma, U_r(x)} h\|_{L^\infty(U_r(x))}
\le C' r \left( \dashint_{U_{9\sigma r}(x)} |h|^{p}\,dm \right)^{1/p}, 
\]
where 
\[
C' = 2^{4+1/p} 3^{-1+ 2s/p} C_D^{1/p} \sigma^{1+s/p} 
\]
\item 
If $p>s$ and $g \in L^{p}(X)$, then 
$f$ has a $(1-(s/p))$-H\"older continuous representative. 
Furthermore, after taking a continuous representative of $f$, for $x, y \in X$, we have 
\[
|f(x)-f(y)| \le C'' |x,y|^{1-(s/p)} \mathrm{diam}(X)^{s/p} m(X)^{-1/p} \|g\|_{p}, 
\] 
where 
\begin{align*}
C'' &= 4 \cdot 3^{-s/p} CC'C_D^{2/p} \sigma^{-s/p}  \\
&= 2^{6+1/p} 3^{1-s/p} C_D^{1+3/p} (1+C_D) C_P.
\end{align*}
\end{enumerate}
\end{theorem}
\begin{proof}
Let $y \in U_r(x)$ be a Lebesgue point of $f$ and $i_0$ the least integer with $2^{i_0} \ge 2 \sigma \mathrm{diam}(U_r(x))$. 
Then, as in the proof of \cite[Theorem 5.2]{HK}, we have 
\begin{align*}
\left| f(y) - \dashint_{U_r(x)} f\, dm \right| 
&\le C_D (1+2^s) C_P \sum_{i=-\infty}^{i_0} 2^i \sigma^{-1} \left( \dashint_{B_i(x)} g^{p} \,dm\right)^{1/p} \\
&\le C_D (1+2^s)C_P \sigma^{-1} (J_{p}^{\sigma, U_r(x)} g)(y). 
\end{align*}
This completes the proof of the first statement.

Let us assume $p > s$. 
Let $x \in X$, $r > 0$ and $y \in U_r(x)$ be fixed. 
Let $i_0$ be the least integer such that $2^{i_0} \ge 2 \sigma \mathrm{diam}(U_r(x))$. 
Then, we have
\[
B_{i_0}(y) \subset U_{9 \sigma r} (x). 
\]
Using \eqref{eq:volume comp}, 
we have 
\begin{align*}
J_{p}^{\sigma, U_r(x)} h(y) &= \sum_{2^{i} \le 2 \sigma \mathrm{diam}(U_r(x))} 2^{i} \left( \dashint_{B_i(y)} |h|^{p}\,dm \right)^{1/p} \\
&\le 
\sum_{i=-\infty}^{i_0} 2^{i} \left( \frac{m(U_{9 \sigma r}(x))}{m(B_i(y))} \right)^{1/p} \left( \dashint_{U_{9 \sigma r}(x)} |h|^p\, dm \right)^{1/p} \\
&\le 
2^{1/p} 9^{s/p} C_D^{1/p} \sigma^{s/p} \sum_{i=-\infty}^{i_0} 2^{i(1-(s/p))} r^{s/p} \left( \dashint_{U_{9 \sigma r}(x)} |h|^p\, dm \right)^{1/p}.
\end{align*}
Here, we remark that
\begin{align*}
\sum_{i=-\infty}^{i_0} 2^{i(1-(s/p))} &= 
\frac{(2^{1-s/p})^{i_0}}{1-2^{-1+s/p}} = 
(2^{i_0-1})^{1-s/p} \frac{2^{1-s/p}}{1-2^{-1+s/p}}\\
&\le 
(4 \sigma r)^{1-s/p} \frac{2^{1-s/p}}{1-2^{-1+s/p}} \le 
\frac{16 \sigma}{3} r^{1-s/p}.
\end{align*}
Therefore, we obtain (2). 

We prove (3). 
Let $(f,g)$ be as in the assumption. 
Suppose $p>s$ and $g \in L^p(X)$. 
Let $D$ be a countable dense set in $X$ and $\mathcal B$ a family of open balls of rational radii centered at points of $D$, that is, $\mathcal B = \{U_{r}(x)\}_{x \in D, r \in \mathbb Q_{>0}}$. 
Let us denote $\mathcal B$ by $\mathcal B = \{E_i\}_{i=1}^\infty$ and $E_i$ by $E_i = U_{r_i}(x_i)$. 
Note that $\bigcup_{i=1}^\infty E_i =X$. 
By (1), there exists $A_i \subset E_i$ with $m(E_i \setminus A_i) = 0$ such that 
\[
\left| f(x) - \dashint_{E_i} f\,dm \right| \le C (J_p^{\sigma, E_i} g)(x)
\]
holds for every $x \in A_i$. 
Let us set 
\[
A := X \setminus \bigcup_{i=1}^\infty (E_i \setminus A_i). 
\]
Then, we have $m(X \setminus A) = 0$. 
For $x, y \in A$, we set $r := |x,y|$ and take $i$ with $|x,x_i|<r/2$ and $3r/2<r_i<2r$.
Then, $x, y \in E_i$ and hence, $x, y \in A_i$. 
Therefore, we have 
\begin{align*}
|f(x)-f(y)| &\le \left| f(x) - \dashint_{E_i} f\,dm \right| + \left| f(y) - \dashint_{E_i} f\,dm \right| \\
&\le C (J_p^{\sigma, E_i} g(x) + J_p^{\sigma, E_i} g(y)) \\
&\le 2 CC' r_i \left( \dashint_{U_{9\sigma r_i}(x_i)} g^p\,dm\right)^{1/p}\\
&\le 4CC' r \left( \dashint_{U_{9\sigma r_i}(x_i)} g^p\,dm\right)^{1/p}. 
\end{align*}
Here, we estimate the last factor. We denote $U_a(x_i)$ by $U_a$ in the following.
\begin{align*}
\dashint_{U_{9\sigma r_i}(x_i)} g^p\,dm &\le \frac{m(U_{18 \sigma r})}{m(U_{3 \sigma r})} \dashint_{U_{18 \sigma r}} g^p\,dm \\
&\le C_D 6^s \frac{1}{m(U_{18 \sigma r})} \|g\|_p^p \\
&\le C_D^2 6^s \left( \frac{\mathrm{diam}(X)}{18\sigma r} \right)^s  \frac{1}{m(X)} \|g\|_p^p. 
\end{align*}
Hence, we obtain the conclusion of (3).
\end{proof}

\begin{proof}[Proof of Theorem \ref{thm:HK}]
Let $(X,m)$ be as in Theorem \ref{thm:HK}. 
Let $C_D$, $p_0$, $C_P$ and $\sigma$ be constant appeared in \eqref{eq:doubling} and \eqref{eq:Poincare}. 
If $p \ge p_0$, then $W^{1,p} \subset W^{1,p_0}$. 
Moreover, by the H\"older inequality, for $f \in W^{1,p}(X,m)$, the $p$-Poincar\'e inequality holds in the sense that 
\[
\dashint_{U_r(x)} \left| f - \dashint_{U_r(x)} \,dm \right|\, dm \le C_P r \left( \dashint_{U_{\sigma r}(x)} |Df|_\ast^p\,dm \right)^{1/p}
\]
for arbitrary $x \in X$ and $r > 0$, 
where $C_P$ and $\sigma$ are the same constants as those of \eqref{eq:Poincare}. 
Therefore, if $p > s = \log_2 C_D$, by Theorem \ref{thm:HK2} and by $m(X)=1$, $f$ is $(1-s/p)$-H\"older continuous and 
\[
\frac{|f(x)-f(y)|}{|x,y|^{1-s/p}} \le C'' \mathrm{diam}(X)^{s/p}\| |Df|_\ast\|_p
\]
holds, where $C''$ is the same as that in Theorem \ref{thm:HK2} (3). 
So, the desired constant $C(p) = C'' \mathrm{diam}(X)^{s/p}$ is uniformly bounded whenever $p > \max\{ s, p_0 \}$. 
This completes the proof. 
\end{proof}

We now verify \eqref{eq:Morrey} with uniform estimate \eqref{eq:uniform Holder} in the case that $M$ is a closed Riemannian manifold. 
Let $\kappa$ be a lower bound of the Ricci curvature of $M$ and $n=\dim M$. 
Due to Bishop-Gromov inequality, $(M, \mathrm{vol}_g)$ is doubling.  
Indeed, 
\[
\mathrm{vol}_g(U_{2r}(x)) \le 2^n \exp(\sqrt{-(n-1) \kappa} r) \mathrm{vol}_g(U_{r}(x))
\]
holds for every $x \in M$ and $r > 0$. 
So, the doubling constant $C_D$ as in \eqref{eq:doubling} is given by a constant depending only on $n, \kappa$ and $d = \mathrm{diam}(M)$. 
We may assume that $\kappa \ge 0$. 
Then, by Buser's inequality (\cite{Bu}), we know that 
\[
\int_{U_r(x)} \left| f - \dashint_{U_r(x)} f\,d \mathrm{vol}_g \right|\,d \mathrm{vol}_g \le C(n) \exp(\sqrt{-\kappa} r)r \int_{U_r(x)} |\nabla f|\,d \mathrm{vol}_g
\]
holds for every $x \in M$ and $r > 0$, where $C(n)$ is a constant depending only on $n$.
So, the $1$-Poincar\'e inequality holds in the above sense, and the constant $C_P$ is given by a constant depending only on $\kappa, n$ and $d$. 
Therefore, by Theorem \ref{thm:HK}, we know that \eqref{eq:Morrey} is true together with a uniform estimate \eqref{eq:uniform Holder}.

\end{document}